\documentclass[a4paper]{article}
\usepackage[utf8x]{inputenc}
\usepackage{amsmath,amsfonts,amsthm,graphicx,enumerate,float,enumitem,listings,makeidx,hyperref}
\usepackage[toc,page]{appendix}
\usepackage{xcolor}
\makeindex

\newtheorem{theorem}{Theorem}[section]
\newtheorem{lemma}[theorem]{Lemma}
\newtheorem{proposition}[theorem]{Proposition}
\theoremstyle{definition}
\newtheorem{definition}[theorem]{Definition}
\theoremstyle{remark}
\newtheorem{remark}[theorem]{Remark}

\newcommand{\EE}{\mathcal{E}}
\newcommand{\TEE}{\tilde{\mathcal{E}}}
\newcommand{\EIT}{\mathcal{E}}
\newcommand{\QQ}{\mathbb{Q}}
\newcommand{\NN}{\mathbb{N}}

\title{Tanaka-Ito $\alpha$-continued fractions and matching}
\author{
\textsc{Carlo Carminati} \\
\normalsize Dipartimento di Matematica, Universit\`a di Pisa, \\
\normalsize Largo Bruno Pontecorvo 5, 56127 Pisa, Italy \\
\normalsize \url{carlo.carminati@unipi.it}
\and 
\textsc{Niels Langeveld} \\
\normalsize Department of Mathematics, Leiden University, \\
\normalsize Niels Bohrweg 1, 2333CA Leiden, The Netherlands \\
\normalsize \url{n.d.s.langeveld@math.leidenuniv.nl}
\and 
\textsc{Wolfgang Steiner} \\
\normalsize Universit\'e de Paris, IRIF, CNRS, 75013 Paris, France \\
\normalsize \url{steiner@irif.fr}}

\begin{document}
\maketitle

\begin{abstract}
Two closely related families of $\alpha$-continued fractions were introduced in 1981: by Nakada on the one hand, by Tanaka and Ito on the other hand. 
The behavior of the entropy as a function of the parameter $\alpha$ has been studied extensively for Nakada's family, and several of the results have been obtained exploiting an algebraic feature called \emph{matching}.
In this article we show that matching occurs also for Tanaka-Ito $\alpha$-continued fractions, and that the parameter space is almost completely covered by matching intervals. Indeed, the set of parameters for which the matching condition does not hold, called bifurcation set, is a zero measure set (even if it has full Hausdorff dimension). 
This property is also shared by Nakada's $\alpha$-continued fractions, and yet there also are some substantial differences: not only does the bifurcation set for Tanaka-Ito continued fractions contain infinitely many rational values, it also contains  numbers with unbounded partial quotients.
\end{abstract}

\section{Introduction}
Several variants of the regular continued fraction (RCF) have been considered;
the most famous ones are the \emph{nearest integer} continued fraction (NICF) and the \emph{backward} continued fraction (BCF). Starting from the 80s, some attention has been devoted to families of continued fraction algorithms; even if different authors have focused on different families, one can describe most\footnote{Actually some authors, e.g.\ in~\cite{LM}, studied the so called \emph{folded algorithms} which are not of the type \eqref{eq:alphamap}, however from the metric viewpoint there is hardly any difference between the folded and the unfolded version; see \S~3.1 of \cite{BDV} for a discussion of this issue.} of these families using the same setting as follows.
Define $T_{\alpha}:[\alpha-1,\alpha]\rightarrow[\alpha-1,\alpha]$ by
\begin{equation}\label{eq:alphamap}
T_{\alpha}(x)= 
\begin{cases} S(x)-\left\lfloor S(x)+1-\alpha\right\rfloor & \text{for } x\neq0, \\
  0 & \text{for } x=0.
  \end{cases}
\end{equation} 
Different choices of~$S$ in formula \eqref{eq:alphamap} give rise to different generalizations of the classical continued fraction algorithms\footnote{The choice $S(x)=-\frac{1}{|x|}$ isomorphic to the case (N), up to exchanging $\alpha$ and $1-\alpha$.}
\begin{enumerate}
\item[(N)] for $S(x)=\frac{1}{|x|}$, one gets the $\alpha$-continued fractions first studied by Nakada \cite{N},
\item[(KU)] for $S(x)=-\frac{1}{x}$, one finds a family of $(a,b)$-continued fractions (with $b=\alpha=a+1$), which were first studied by Katok and Ugarcovici \cite{KU1},
\item[(TI)] for $S(x)=\frac{1}{x}$, one gets the $\alpha$-continued fractions first studied by Tanaka and Ito \cite{TI}. 
\end{enumerate}
In all the above three cases, for all $\alpha \in (0,1)$, the dynamical system defined by the map \eqref{eq:alphamap} admits an absolutely continuous invariant probability measure~$\mu_\alpha$ (see \cite{N}, \cite{KU3} and \cite{NS} respectively) allowing the study of the metric entropy $h_{\mu_\alpha}(T_\alpha)$. This determines the speed of convergence of the continued fraction algorithm on typical points. An issue which has been in the spotlight in recent years is the dependence of the entropy on the parameter~$\alpha$. The behavior of the entropy is by now quite well understood in case (N), which is by far the most studied \cite{N, CMM, LM, NN, CT1, KSS, CT2}. The same is true for the case (KU), which was first considered much more recently \cite{KU1, KU2, CIT}. 
However, for the case (TI)  the picture is still not complete, and quite a lot of time elapsed between the original results dating back to 1981 \cite{TI} and the recent paper \cite{NS}.

One common feature of all these families is the presence of a property called \emph{matching},\footnote{Definitions of matching vary slightly from article to article, and even the terminology may change: all the terms \emph{matching property}, \emph{cycle property}, \emph{synchronization property} are just different names for the same feature.} which affects both  the behavior of the entropy and the structure of the natural extension. A~parameter $\alpha\in [0,1]$ satisfies the \emph{matching condition} with \emph{matching exponents} $M,N$ if
\begin{equation}\label{eq:matching}
T_{\alpha}^M(\alpha-1)=T_{\alpha}^N(\alpha).
\end{equation}
Actually in all three cases (N), (KU) and (TI), a condition like \eqref{eq:matching} holds  on intervals with non-empty interior\footnote{The explanation for this surprising feature is that each matching interval actually relates to an algebraic identity; this fact will be implicit in our discussion, but details can be found in \cite{L}.}; thus what will be relevant is the definition of a \emph{matching interval}.

\begin{definition}[Matching]\label{def:match-int}
Let $J \subset [0,1]$ be a non-empty open interval. 
We say that $J$ is a \emph{matching interval} (with exponents $M,N$) if $T_{\alpha}^M(\alpha-1) = T_{\alpha}^N(\alpha)$ for all $\alpha \in J$, $T_{\alpha}^{M-1}(\alpha-1) \ne T_{\alpha}^{N-1}(\alpha)$ for almost all $\alpha \in J$, and $J$ is not contained in a larger open interval with these properties.
The difference $\Delta :=M-N$ is called \emph{matching index}.
\end{definition}

\begin{figure}[ht]
\centering
\includegraphics[scale=.6]{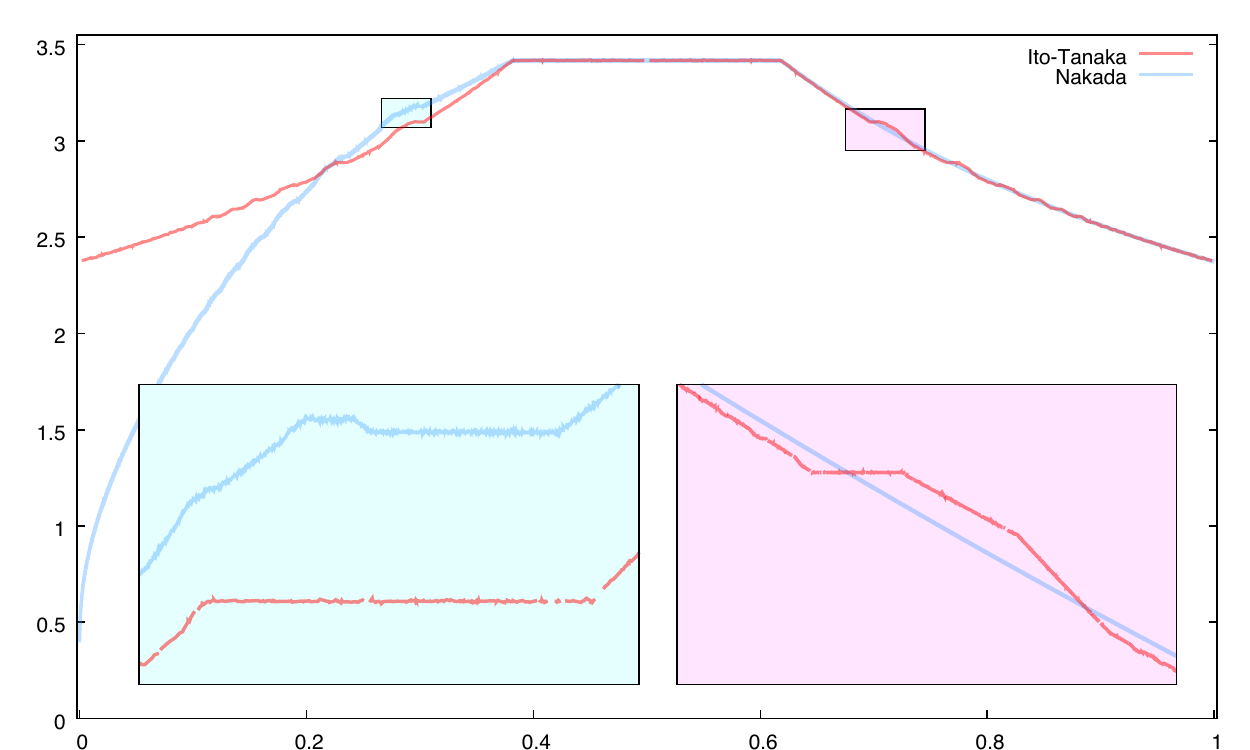}
\caption{Comparison between the entropy functions of the families  (TI) and~(N). The two entropy functions coincide on the interval $[1-g,g]$, and are very close on $[g,1]$ (and yet different, see zoom on the right). On the interval $[0,1-g]$ the entropy of the family (TI) seems monotone, but this is certainly not the case for the family (N) (see zoom on the left) due to the presence of matching intervals of all integer indexes.}\label{fig:comparison}
\end{figure}

Given any matching interval~$J$, one can prove that the behavior of the entropy function $\alpha \mapsto h_{\mu_\alpha}(T_\alpha)$ is strongly correlated to the matching index: if $M,N$ are the matching exponents on~$J$,  then the entropy function is constant when $M=N$, while it is decreasing when $N<M$ and increasing when $N>M$.
This relation was first discovered for the case (N) in the seminal paper \cite{NN}, and was later used in \cite{CIT} to prove a monotonicity result for the case~(KU) (where the property of matching had  already been detected by \cite{KU1} in connection with natural extensions);  for a proof covering all three families at once, see Theorem~3.2.8 in \cite{L}. 

In this paper we shall  focus our study on the matching property for the family~(TI).
Other aspects, such as the relation between the entropy and the natural extension of~$T_\alpha$ are studied in~\cite{NS}.  

We call \emph{matching set} the union of all matching intervals; its complement will be called \emph{bifurcation set} and will be denoted by~$\EE$. The following lemma shows that two matching intervals cannot overlap, and that the matching index is actually well defined.

\begin{lemma} \label{l:overlap}
Let $M, M', N, N'$ be such that $M-N \ne M'-N'$.
Then there are at most countably many $\alpha \in [0,1]$ such that $T_\alpha^M(\alpha-1) = T_\alpha^N(\alpha)$ and $T_\alpha^{M'}(\alpha-1) = T_\alpha^{N'}(\alpha)$.
\end{lemma}

\begin{proof}
Assume w.l.o.g.\ that $N' \ge N$. 
Then we have $T_\alpha^{M+N'-N}(\alpha-1) = T_\alpha^{N'}(\alpha) = T_\alpha^{M'}(\alpha-1)$.
Since $M-N \ne M'-N'$, this implies that $\alpha$ is a rational or quadratic number.
\end{proof}

Let us point out that Lemma \ref{l:overlap} applies to all the above three cases (TI), (KU) and (N). In fact, in all these three cases one can easily detect several matching intervals and  also other results are analogous under many aspects.

It is clear, from  the definition we chose,  that  matching  is an open condition. For the $\alpha$-continued fractions~(N), it is shown in~\cite{CT1} that matching holds almost everywhere. The same is true in the case (KU); see \cite{KU1, KU2, CIT}. In Section~2, we will show that this is also true for the $\alpha$-continued fractions of Tanaka and Ito. However, when we come to the bifurcation set, the situation is different. Not only do each of the three variants (N), (KU) and (TI) have a different bifurcation set but these bifurcation sets display quite a few differences. For instance, it is not difficult to show that the bifurcation set of the (N) case and the (KU) case both do not intersect $\QQ\cap(0,1)$ and are made of \emph{constant type} numbers (numbers for which the digits of the continued fraction expansion are bounded from above).
This is not the case for (TI): not only does the bifurcation set contain infinitely many rational values (such as $1/n$ for $n \geq 3$) but it also contains numbers with unbounded partial quotients.

In this paper we will focus on  the behavior of matching for Tanaka-Ito $\alpha$-continued fractions; in the following subsection we provide some background information for this particular case and we state our results. 

\subsection{Tanaka-Ito continued fractions: old and new results} \label{sec:tanaka-ito-continued}
In the following, $T_\alpha$ will always denote the map \eqref{eq:alphamap} for the Tanaka-Ito case, i.e., with $S(x)=1/x$. 
Let us point out that the dynamical systems of $\alpha$ and $1-\alpha$ are isomorphic. Indeed, setting $\tau(x)=-x$ gives 
\begin{equation}\label{eq:symmetry}
\tau \circ T_{\alpha}= T_{1-\alpha} \circ \tau
\end{equation}
for almost every $x\in[\alpha-1,\alpha]$, the only exceptions being the discontinuity points of $T_\alpha$. This implies that the entropy is symmetric with respect to the point $\alpha=1/2$, and the same is true for the bifurcation set $\EE$, since it turns out that these exceptions are irrelevant for the definition of matching interval; see Proposition~\ref{p:atmatching}. Thanks to this symmetry we can restrict our study to  the parameter range $\alpha \in [1/2,1]$.

Setting  $d_{\alpha}(x)= \left\lfloor S(x)+1-\alpha\right\rfloor$, for every $x\in[\alpha-1,\alpha]$ we use the shorthand  $d_{\alpha,n}=d_{\alpha,n}(x)=d_{\alpha}\left(T_{\alpha}^n(x)\right)$ 
to write the continued fraction expansion
\begin{equation*}
x=\frac{1}{\displaystyle d_{\alpha,1}(x)+\frac{1}{\displaystyle d_{\alpha,2}(x)+\frac{1}{\ddots}}} \ .
\end{equation*}
Note that $T_1$ is the Gauss map and $T_{\frac{1}{2}}$ is the map for nearest integer continued fraction expansions. Furthermore $d_n(x)$ is called the $n^{\text{th}}$ partial quotient of~$x$ and can be both negative and positive. Now we define the $n^{\text{th}}$ convergent as
\[
c_{\alpha,n}(x)=\frac{p_{\alpha,n}(x)}{q_{\alpha,n}(x)}= \frac{1}{\displaystyle d_{\alpha,1}(x)+\frac{1}{\displaystyle d_{\alpha,2}(x)+\frac{1}{\displaystyle \ddots+\frac{1}{d_{\alpha,n}(x)}}}} \ .
\]
For the speed of convergence  of (TI) $\alpha$-continued fractions we have
\[
\bigg|x-\frac{p_{\alpha,n}}{q_{\alpha,n}}\bigg|\leq \frac{1}{|q_{\alpha,n}|^2} \le g^n \quad \mbox{with} \quad g:= \frac{\sqrt{5}-1}{2}\;;
\] 
see \cite{TI}.
We can see that the faster $q_{\alpha,n}(x)$ grows the faster the convergence. This is related to the entropy in the following way. 
For fixed $\alpha$, we have 
\[
h(T_{\alpha})=2\lim_{n\rightarrow\infty} \frac{1}{n} \log |q_{\alpha,n}(x)|
\] 
for almost all $x\in [\alpha-1,\alpha]$.
For the regular continued fraction map this relation is fairly known.
The proof in our case can be found in \cite{L} and \cite{NS} and can be used to prove monotonicity on the matching intervals.
Let us recall from \cite{TI} that the symmetric parameter interval $(1-g,g)$ is (almost) covered by the three adjacent matching intervals $(1-g,\sqrt{2}-1)$, $(\sqrt{2}-1,2-\sqrt{2})$ and $(2-\sqrt{2},g)$; so the interesting part of the bifurcation set is in the ranges $[0,1-g]$ and $[g,1]$; since the problem is symmetric with respect to $\alpha=1/2$ (see~\eqref{eq:symmetry}), we can focus on $\EIT \cap [g,1]$.
We will prove the following three characterizations of this set. 

\begin{theorem}\label{thm:charactEIT} 
The bifurcation set on $[g,1]$, with $g = \frac{\sqrt{5}-1}{2}$, is given by
\begin{align}
& \EIT\cap [g,1] \nonumber \\
& \ = \{\alpha \in [g,1]:\, T^n_\alpha(\alpha-1) \le \tfrac{1}{\alpha+1}\ \mbox{and}\ T^n_\alpha(\tfrac{1}{\alpha}-1) \le \tfrac{1}{\alpha+1}\ \mbox{for all}\ n \ge 1\} \label{e:bifurcalpha} \\
& \ = \big\{\alpha \in [g,1]:\, T^n_g(\alpha-1) \ge \alpha-1\ \mbox{and}\ T^n_g(\tfrac{1}{\alpha}-1) \ge \alpha-1\ \mbox{for all}\ n \ge 1\big\} \label{e:bifurcg} \\
& \ = \big\{\alpha \in [g,1]:\, T^n_1(\alpha) \notin (\tfrac{1}{\alpha+1}, \alpha) \ \mbox{for all}\ n \ge 2\ \mbox{and} \label{e:bifurc1} \\
& \hspace{7em} T^n_1(\alpha) \notin (1-\alpha, \tfrac{\alpha}{\alpha+1})\ \mbox{for all}\ n \ge 2\ \mbox{such that $P_n(\alpha)$ is odd}\big\}, \nonumber
\end{align}
where $P_n(\alpha) = \min\{k \ge 0:\, T_1^{n-k}(\alpha) \le 1-\alpha \ \mbox{or}\ T_1^{n-k-1}(\alpha) \ge \alpha\}$.
\end{theorem}

While the characterization in terms of $T_\alpha$ is natural from the definition of the bifurcation set, the characterizations with fixed maps $T_g$ and $T_1$ (which is the classical Gauss map) will be more useful.
In particular, from the ergodicity of $T_g$ and $T_1$ it easily follows that $\EIT$ is a Lebesgue measure zero set. 

\begin{theorem}\label{thm:matchingae}
Matching holds almost everywhere on $[0,1]$ and the only possible indices are $-2,0,2$. 
More precisely, the matching indices are $0$ or $2$ for $\alpha \le 1/2$, and $0$ or $-2$ for $\alpha \ge 1/2$. 
\end{theorem}

The difference with the matching index for the family (N) and (KU) is thus quite evident: indeed the set of all possible matching indexes of itervals contained in $[0,1/2]$ is $\mathbb{Z}$ for the family (N), and $\NN$ for the family (KU).

The following theorems describe the bifurcation set $\EIT$, i.e. the points where the matching property fails.

\begin{theorem}\label{th:aroundg}
We have that $\EIT$ is a Lebesgue measure zero set and $$\dim_H(\EIT)=1.$$  
Moreover, for all $ \delta >0$ 
\begin{equation*}
\dim_H\left(\EIT\cap (g,g+\delta)\right)=1 \quad \mbox{and} \quad \dim_H\left(\EIT\cap (g+\delta, 1)\right)<1.
\end{equation*}
\end{theorem}

\begin{theorem}\label{th:bifrat}
The bifurcation set $\EIT$ contains infinitely many rational values, and the set of rational bifurcation parameters $\EIT\cap \QQ$ has no isolated points. Moreover for all $r\in \EIT\cap \QQ$ and for all $\delta>0$ we have that
\[
\dim_H(\EIT\cap (r-\delta,r+\delta))>1/2.
\]
\end{theorem}

Theorems~\ref{thm:charactEIT} and \ref{thm:matchingae} are proved in Section~\ref{sec:matchae}. In Section~\ref{sec:dim}, we prove the theorems on dimensional results (Theorems~\ref{th:aroundg} and~\ref{th:bifrat}).

\begin{remark}
 In the case (N) of $\alpha$-continued fractions of Nakada the bifurcation set $\EE_N$  admits an even simpler characterization in terms of the Gauss map~$T_1$ (see~\cite{BCIT}):
\begin{equation}\label{eq:EN}
\EE_N=\{\alpha\in [0,1] :  T_1^k(\alpha)\geq \alpha\ \mbox{for all}\ k\in \NN  \}.
\end{equation}
This characterization is useful to spot analogies and differences. On the one hand, one can easily prove that, for all $\delta>0$, $\dim_H(\EE_N \cap [0,\delta])=1$ and $\dim_H(\EE_N \cap [\delta,1])<1$, a result similar to Theorem~\ref{th:aroundg} with $0$ playing the role of~$g$.
On the other hand \eqref{eq:EN} shows that $\EE_N$ only contains constant type numbers, and in particular it does not contain any positive rational value. Let us also mention that equation \eqref{eq:EN} can also be used to get some insight in the local dimension of $\EE_N$, see \cite{CT3}.
\end{remark}

\subsection{Other families satisfying the matching condition}
Matching can be encountered also in other families of one-dimensional expanding maps, but all cases  known so far  fall in one of the following two types: (a)~continued fraction algorithms; (b)~piecewise affine maps. Moreover, in either  case   matching seems to be induced by some algebraic property of the system.
For instance, let alone the families (N),  (KU) and (TI) mentioned before, in \cite{CKS} the authors study the phenomenon of matching for a family of continued fraction algorithms based on a group which is not the modular group, and each matching interval corresponds to a group identity; see also \cite{DKS}.

The first paper where, in an implicit way, matching was observed in the piecewise affine setting is \cite{BORG}.
This was the starting point for the paper \cite{BCMP} that took advantage of the simplicity of this setting to explore matching in depth. 
Since then, matching has been observed also in other families of affine maps \cite{BCK,DKa}.

However, even in the most simple setting in which matching has been detected, which  is represented by generalized $\beta$-transformations (i.e., the parametric family of maps $(T_\alpha)_\alpha$ where $T_\alpha(x)=\beta x +\alpha - \lfloor \beta x +\alpha\rfloor$ for a certain fixed algebraic integer $\beta$),  there are still several open questions, such as proving that matching has full measure for particular values~$\beta$, or determining which $\beta$ are compatible with matching; see \cite{BCK}. 
One curious feature that has been observed in some of the piecewise affine cases is an explicit bijection between the bifurcation set of Nakada's $\alpha$-continued fractions and a segment of  the bifurcation set of the piecewise family considered; see Remark~4.29 in \cite{BCMP}.

\section{Characterizations of matching intervals and the exceptional set}\label{sec:matchae}
The main tool for the proof of Theorem~\ref{thm:charactEIT} is the following technical lemma which can be used both to compare $\alpha$-continued fractions of two numbers (in particular of $\alpha-1$ and $T_\alpha(\alpha) = \frac{1}{\alpha}-1$) as well as to translate an $\alpha$-continued fraction into a $\beta$-continued fraction (in particular for $\beta = 1$). 

\begin{lemma} \label{l:alphabeta}
Let $g \le \alpha \le \beta \le 1$, $x \in [\alpha-1,\alpha)$, $y \in [\beta-1,\beta)$. 
\begin{enumerate}[label=(\roman*)]
\item \label{i:alphabeta1}
If $x=y$, then $T_\beta(y) - T_\alpha(x) \in \{0,1\}$.
\item \label{i:alphabeta4}
If $y-x =1$, then $(x+1) (T_\beta(y)+1) = 1$.
\item \label{i:alphabeta2}
If $(x+1)(y+1) = 1$ or $x+y = 0$, then $T_\alpha(x) + T_\beta(y) \in \{0,1\}$.
\item \label{i:alphabeta3}
If $x+y = 1$, then 
\[
\begin{cases}  
T_\beta(y) - T_\alpha^2(x) \in \{0,1\} & \text{if } x > \frac{1}{\alpha+1}, 
\\  
 T_\beta^2(y) - T_\alpha(x) \in \{0,1\} & \text{if } y > \frac{1}{\beta+1},
\\
\big(T_\alpha(x)+1\big) \big(T_\beta(y)+1\big) = 1 & \text{otherwise}.
\end{cases}
\] 
\end{enumerate}
\end{lemma}

\begin{figure}[ht]
\centering
\includegraphics[scale=.4]{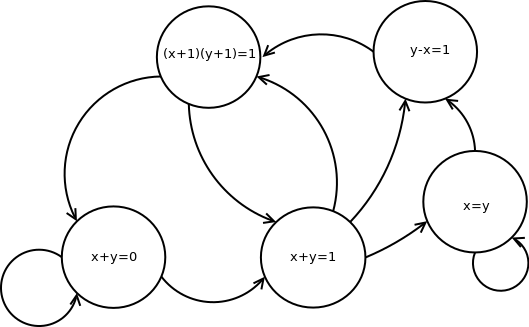}
\caption{A diagram for Lemma~\ref{l:alphabeta}.}\label{fig:diag}
\end{figure}

In Figure~\ref{fig:diag}, one can see which condition can imply which other condition.

\begin{proof}
\textbf{Case \ref{i:alphabeta1}.}
We have $T_\beta(y)-T_\alpha(x) \in \mathbb{Z} \cap (\beta-\alpha-1, \beta-\alpha+1) = \{0,1\}$.

\smallskip\noindent
\textbf{Case \ref{i:alphabeta4}.} 
Since $x \ge \alpha-1$, we have $y \ge \alpha$, thus $(x+1) (T_\beta(y)+1) = \frac{x+1}{y} = 1$. 

\smallskip\noindent
\textbf{Case \ref{i:alphabeta2}.}
Dividing the equations by $xy$ gives us $\frac{1}{x} + \frac{1}{y} = -1$ and $\frac{1}{x} + \frac{1}{y} = 0$ respectively. 
This implies that $T_\alpha(x) + T_\beta(y) \in \mathbb{Z} \cap [\alpha+\beta-2, \alpha+\beta) = \{0,1\}$.

\smallskip\noindent
\textbf{Case \ref{i:alphabeta3}.} 
If $x > \frac{1}{\alpha+1}$, then $\frac{1}{T_\alpha(x)} = \frac{1}{\frac{1}{x}-1}  =\frac{x}{1-x}=\frac{1-y}{y}= \frac{1}{y} - 1$, thus $\frac{1}{y}-\frac{1}{T_\alpha(x)} =1$ and so $T_\beta(y) - T_\alpha^2(x) \in \mathbb{Z} \cap (\beta-1-\alpha,\beta-\alpha+1) = \{0,1\}$. 
Similarly, $y > \frac{1}{\beta+1}$ implies that $T_\beta^2(y) - T_\alpha(x) \in \{0,1\}$. 
\newline 
If $x \le \frac{1}{\alpha+1}$ and $y \le \frac{1}{\beta+1}$, then $x = 1-y \ge \frac{\beta}{\beta+1}\ge \frac{g}{g+1}= \frac{1}{g+2} \ge \frac{1}{\alpha+2}$ and $y = 1-x \ge \frac{\alpha}{\alpha+1} \ge \frac{1}{g+2} \ge \frac{1}{\beta+2}$.
We cannot have $x = \frac{1}{\alpha+2}$ because this would imply that $\alpha = g = \beta = y$, contradicting that $y < \beta$. 
Similarly, we cannot have $y = \frac{1}{\beta+2}$.
From $x \in (\frac{1}{\alpha+2}, \frac{1}{\alpha+1}]$ and $y \in (\frac{1}{\beta+2}, \frac{1}{\beta+1}]$, we infer that $(T_\alpha(x)+1) (T_\beta(y)+1) = (\frac{1}{x} - 1) (\frac{1}{y} - 1) = 1$. 
\end{proof}
 
Lemma~\ref{l:alphabeta} greatly simplifies when taking $\alpha=\beta$ and only looking at the orbits of $\alpha-1$ and $\frac{1}{\alpha}-1$ before exceeding $\frac{1}{\alpha+1}$.
We use the notation
\begin{equation}\label{eq:xix}
x_n:=T_\alpha^n(\alpha-1), \qquad y_n:=T_\alpha^n(\tfrac{1}{\alpha}-1).
\end{equation}

\begin{lemma} \label{lem:beforematching}
Let $\alpha \in [g,1]$ and $m\in\mathbb{N}$ be such that 
\begin{equation} \label{e:xynle}
x_n \leq \tfrac{1}{\alpha+1} \quad \mbox{and} \quad  y_n \leq \tfrac{1}{\alpha+1} \quad \mbox{for all}\ 0 \le n<m.
\end{equation}
Then for all $0\leq n\leq m$ the pair $(x_n,y_n)$ satisfies one of the following relations:
\[
\begin{array}{cc}
\mathrm{(A)} & (x_n +1)(y_n +1) = 1,\\
\mathrm{(B)} & x_n +y_n = 0,\\
\mathrm{(C)} & x_n +y_n = 1.
\end{array}
\]
If $x_m > \frac{1}{\alpha+1}$ or $y_m > \frac{1}{\alpha+1}$, then $x_m+y_m=1$.
\end{lemma}

\begin{figure}[ht]
\centering
\includegraphics[scale=.35]{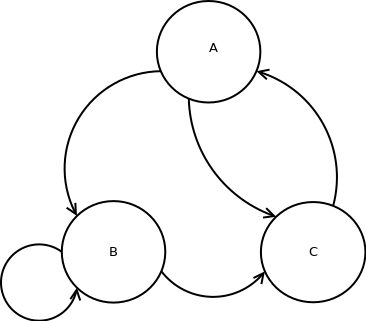}
\caption{A diagram for Lemma~\ref{lem:beforematching}.}\label{fig:diagsimple}
\end{figure}

In Figure~\ref{fig:diagsimple} one can see from which state to which state you can get.

\begin{proof}
The proof is a straightforward  application of Lemma~\ref{l:alphabeta}.
The pair $(x_0,y_0)$ satisfies~(A), condition (i) in Lemma~\ref{l:alphabeta}. Let $ 0 \le n<m$ then $y_n-x_n=1$ is impossible since $x_n,y_n\in[\alpha-1,\alpha)$. Also $x_n=y_n$ is impossible since we have $x_n \le \frac{1}{\alpha+1}$ and $y_n \le \frac{1}{\alpha+1}$ which implies that $(x_n,y_n)$ always are in state (A), (B) or (C).
We find that if $(x_n,y_n)$ satisfies (A) or~(B), then $(x_{n+1},y_{n+1})$ satisfies (B) or~(C). 
If $(x_n,y_n)$ satisfies (C), then (A) holds for $(x_{n+1},y_{n+1})$. 

Now suppose that $x_m > \frac{1}{\alpha+1}$ and (B) holds. Then $y_m<-\frac{1}{\alpha+1}<\alpha-1$ which contradicts with $y_m\in[\alpha-1,\alpha)$. If $x_m > \frac{1}{\alpha+1}$ and (A) holds we find  $y_m=\frac{1}{x_m+1}-1<\frac{1}{\frac{1}{\alpha+1}+1}-1=-\frac{1}{2+\alpha}<\alpha-1$ since $\alpha>g$ which also contradicts with $y_m\in[\alpha-1,\alpha)$. Note that the role of $x_m$ and $y_m$ are interchangeable. We find that if 
 $x_m > \frac{1}{\alpha+1}$ or $y_m > \frac{1}{\alpha+1}$, then $x_m+y_m=1$.
\end{proof}
We focus now on the complement of the set
\[
\TEE = \{\alpha \in [g,1]:\, x_n \le \tfrac{1}{\alpha+1}\ \mbox{and}\ y_n \le \tfrac{1}{\alpha+1}\ \mbox{for all}\ n \ge 1\}
\]
(which is the set in~\eqref{e:bifurcalpha}) and show that it is a union of matching intervals. We also prove certain useful properties of these matching intervals.
 
\begin{proposition} \label{p:atmatching}
Let $\alpha \in (g,1]$, $m \in \mathbb{N}$ and $\epsilon \in \{-1,1\}$ be such that \eqref{e:xynle} holds and $T_\alpha^m(\alpha^\epsilon-1) >  \frac{1}{\alpha+1}$.
Then $\alpha$ belongs to a matching interval~$J$ with the following properties:
\begin{enumerate}[label=(\roman*)]
\item
the matching exponents are $M = m+2-\frac{1-\epsilon}{2}$ and $N = m+2+\tfrac{1-\epsilon}{2}$,
\item
for all $\alpha \in J$, the inequalities in \eqref{e:xynle} and $T_\alpha^m(\alpha^\epsilon-1) >  \frac{1}{\alpha+1}$ hold, \\ in particular $\tilde{\mathcal{E}} \cap J = \emptyset$,
\item
the endpoints $\alpha_1$ and $\alpha_2$ of~$J$ satisfy $T_{\alpha_1}^m(\alpha_1^\epsilon-1) = \frac{1}{\alpha_1+1}$ and \\ $\lim_{z\in J,z\to \alpha_2} T_z^m(z^\epsilon-1) = \alpha_2$, 
\item
for all $\alpha \in J$, we have $T_\alpha^n(\alpha) \neq \alpha-1$ for all $1\leq n<N$ and \\ $T_\alpha^n(\alpha-1) \neq \alpha-1$ for all $1\leq n<M$.
\end{enumerate}
\end{proposition}

\begin{remark}\label{r:EITinEE}
Proposition~\ref{p:atmatching} implies that $\EIT \cap [g,1] \subset \TEE$.
\end{remark}
For the proof of the proposition, we use the following lemma. 

\begin{lemma} \label{l:continuous}
Let $\alpha \in (g,1]$, $m \in \mathbb{N}$ such that \eqref{e:xynle} holds, and $x_m > \frac{1}{\alpha+1}$ or $y_m > \frac{1}{\alpha+1}$. 
Then the maps $T_z^m(z-1)$ and $T_z^m(\frac{1}{z}-1)$ are continuous at $z = \alpha$, and all inequalities in \eqref{e:xynle} are strict.
\end{lemma}

\begin{proof}
The maps $T_z^n(z-1)$ and $T_z^n(\frac{1}{z}-1)$ are continuous at $z = \alpha$ for all $1 \le n \le m$ if and only if $x_n \ne x_0$ and $y_n \ne x_0$ for all $1 \le n \le m$.
Suppose that $x_n = x_0$ or $y_n = x_0$ for some $1 \le n \le m$. If $(x_n,y_n)$ satisfies (C) and  $x_n = x_0$ then $x_n+y_n=\alpha-1+y_n=1$ and so $y_n=2-\alpha>\alpha$. Since we can use the same reasoning for $y_n=x_0$ we find that $(x_n, y_n)$ satisfies (A) or (B).  This gives $n < m$ and $x_{n+1} + y_{n+1} \in \{0,1\}$, $x_1 + y_1 \in \{0,1\}$. We find $x_{n+1} + y_{n+1}-(x_1 + y_1)\in\{-1,0,1\}$. If $x_n = x_0$, then  we have $x_{n+1} = x_1$ and $x_{n+1} + y_{n+1}-(x_1 + y_1)=y_{n+1}-y_1\in\{-1,0,1\}$ where we can exclude $\pm 1$ since $y_{n+1},y_1\in[\alpha-1,\alpha)$ and thus $y_{n+1} = y_1$; if $y_n = x_0$, then we have $y_{n+1} = x_1$ and thus $x_{n+1} = y_1$. We get that $\{x_{m-n}, y_{m-n}\} = \{x_m, y_m\}$, contradicting~\eqref{e:xynle}. 

Since $x_n = \frac{1}{\alpha+1}$ and $y_n = \frac{1}{\alpha+1}$ imply $x_{n+1} = x_0$ and $y_{n+1} = x_0$ respectively, all inequalities in \eqref{e:xynle} are strict.
\end{proof}

\begin{proof}[Proof of Proposition~\ref{p:atmatching}] 
Let $\alpha_0 \in (g,1]$, $m \in \mathbb{N}$ and $\epsilon \in \{-1,1\}$ be such that \eqref{e:xynle} holds for $\alpha_0$ and $T_{\alpha_0}^m(\alpha_0^\epsilon-1) >  \frac{1}{\alpha_0+1}$.
By Lemma~\ref{lem:beforematching}, we have $x_m + y_m = 1$.
Then Lemma~\ref{l:alphabeta} gives that $x_{m+2} = y_{m+1}$ if $x_m > \frac{1}{\alpha_0+1}$, i.e., $T_{\alpha_0}^{m+2}(\alpha_0-1) = T_{\alpha_0}^{m+2}(\alpha_0)$, and that $x_{m+1} = y_{m+2}$ if $y_m > \frac{1}{\alpha_0+1}$, i.e., $T_{\alpha_0}^{m+1}(\alpha_0-1) = T_{\alpha_0}^{m+3}(\alpha_0)$.
By Lemma~\ref{l:continuous}, the maps $T_z^n(z-1)$ and $T_z^n(\frac{1}{z}-1)$ are continuous at $z = \alpha_0$ for all $1 \le n \le m$, and all involved inequalities are strict.
Therefore, $\alpha_0$ is in the interior of a matching interval~$J$ with matching exponents $M = m+2-\frac{1-\epsilon}{2}$ and $N = m+2+\tfrac{1-\epsilon}{2}$.

Let $f$ be the linear fractional transformation satisfying $f(z) = T_z^m(z^\epsilon-1)$ around $z = \alpha_0$. 
By Lemma~\ref{l:continuous}, we get for all $\alpha$ satisfying $\frac{1}{\alpha+1} < f(\alpha) < \alpha$ that $f(\alpha) = T_\alpha^m(\alpha^\epsilon-1)$ and \eqref{e:xynle} holds. 
Since $f$ is expanding at these points, we have some $\alpha_1$, $\alpha_2$ with $f(\alpha_1) = \frac{1}{\alpha_1+1}$ and $f(\alpha_2) = \alpha_2$.
Then $J$ contains the open interval with endpoints $\alpha_1,\alpha_2$.
Arbitrarily close to $\alpha_1$ and $\alpha_2$, we can find $\alpha$ where the minimal $n$ such that $T_\alpha^n(\alpha-1) \ge \frac{1}{\alpha+1}$ or $T_\alpha^n(\frac{1}{\alpha}-1) \ge \frac{1}{\alpha+1}$ is different from~$m$.
Therefore, these points are in matching intervals with different matching exponents than~$J$. 
Hence, by Lemma~\ref{l:overlap}, they are not in~$J$, and the endpoints of~$J$ are $\alpha_1$ and~$\alpha_2$.
Since $T_z^m(z^\epsilon-1)$ is continuous on~$J$, we have $T_{\alpha_1}^m(\alpha_1^\epsilon-1) = f(\alpha_1) = \frac{1}{\alpha_1+1}$ and $\lim_{z\in J,z\to \alpha_2} T_z^m(z^\epsilon-1) = f(\alpha_2) = \alpha_2$.

By Lemma \ref{l:continuous}, we have for all $\alpha \in J$ that $x_n \ne x_0$ for all $1 \le n \le m$ and $y_n \ne x_0$ for all $0 \le n \le m$.
If $x_m > \frac{1}{\alpha+1}$, then we also have $x_{m+1} \ne x_0$, and $y_m > \frac{1}{\alpha+1}$ implies that $y_{m+1} \ne x_0$. 
This gives that $T_\alpha^n(\alpha) \neq \alpha-1$ for all $1\leq n<N$ and $T_\alpha^n(\alpha-1) \neq \alpha-1$ for all $1\leq n<M$.
\end{proof}

In order to prove that $\TEE$ has measure zero, we prove that it is equal to the set in \eqref{e:bifurcg}. 

\begin{lemma} \label{lem:X}
Let $\alpha \in (g,1]$, $z \in [\alpha-1, g)$. 
The following conditions are equivalent:
\begin{enumerate}[label=(\roman*)]
\item \label{i:X1}
$T_\alpha^n(z) = T_g^n(z)$ for all $n\in \NN$,
\item \label{i:X2}
$T_g^n(z) \geq \alpha-1$ for all $n\in \NN$,
\item \label{i:X3}
$T_\alpha^n(z) < g$ for all $n\in \NN$,
\item \label{i:X4}
$T_\alpha^n(z) \leq \frac{1}{\alpha+1}$ for all $n\in \NN$.
\end{enumerate}
In particular, we have 
\[
\TEE = \big\{\alpha \in [g,1]:\, T^n_g(\alpha-1) \ge \alpha-1\ \mbox{and}\ T^n_g(\tfrac{1}{\alpha}-1) \ge \alpha-1\ \mbox{for all}\ n \ge 1\big\}.
\]
\end{lemma}

\begin{proof}
The equivalences \ref{i:X2} $\Leftrightarrow$ \ref{i:X1} $\Leftrightarrow$ \ref{i:X3} are direct consequences of the definition of~$T_\alpha$.
Since $\frac{1}{1+\alpha} < g$, we have \ref{i:X4} $\Rightarrow$ \ref{i:X3}. 
For the converse, suppose that $T_\alpha^n(z) > \frac{1}{\alpha+1}$ for some~$n$. 
Then we have $T_\alpha^{n+1}(z) = \frac{1}{T_\alpha^n(z)} - 1$, thus $T_\alpha^n(z) \geq g$ or $T_\alpha^{n+1}(z) > \frac{1}{g} - 1 = g$, hence \ref{i:X3} does not hold.
\end{proof}

Now we prove that matching is prevalent and the only indices are $-2,0,2$.

\begin{proof}[Proof of Theorem~\ref{thm:matchingae}]
We claim that $\TEE$ has measure zero. Indeed, we have
\begin{align*}
\tilde{\mathcal{E}} & \subset \{\alpha \in (g,1]:\, T_g^n(\alpha-1) \ge \alpha-1\ \mbox{for all}\ n \ge 1\} \\
& \subset \bigcup_{k=1}^\infty \{\alpha \in (g,1]:\, T_g^n(\alpha-1) \ge g-1+\tfrac{1}{k}\ \mbox{for all}\ n \ge 1\}.
\end{align*}
Since $T_g$ is ergodic (with respect to an absolutely continuous invariant measure), all these sets have Lebesgue measure zero, and the same is true for~$\TEE$. 
By Proposition~\ref{p:atmatching}, we obtain that the matching set has full measure on $[g,1]$. 
Therefore, Proposition~\ref{p:atmatching} gives all matching intervals in $[g,1]$, hence the only possible indices are $0$ and~$-2$.
By the symmetry mentioned at the beginning of Section~\ref{sec:tanaka-ito-continued}, using Proposition~\ref{p:atmatching}~(iv), we obtain that the matching set also has full measure on $[0,1-g]$, with the only possible indices $0$ and~$2$.

By Proposition~\ref{p:atmatching}, we also know that each $\alpha \in [g,1] \setminus \TEE$ belongs to a matching interval and no $\alpha$ in the matching set is in~$\TEE$, thus $\TEE=\EE \cap [g,1]$. 
\end{proof}

For the characterization in terms of the regular continued fraction we need the following lemma.

\begin{lemma} \label{l:xyz}
Let $\alpha \in (g,1]$, $x, y \in [\alpha-1, \alpha)$, $z \in [0,1)$, $x' = T_\alpha(x)$, $x'' = T_\alpha^2(x)$, $y' = T_\alpha(y)$, $y'' = T_\alpha^2(y)$, $z' = T_1(z)$, $z'' = T_1^2(z)$. 
\begin{enumerate}[label=(\roman*)]
\item \label{i:xyz1}
If $z=x$, $x+y=0$ or $(x+1)(y+1) = 1$, then $z'=x'$, $x'+y' \in \{0,1\}$, or $z'=1+x'$  and $x'+y'=0$. 
\item \label{i:xyz2}
If $z=1+x$ and $x+y=0$, then $z''=y'$ and $x'+y' \in \{0,1\}$, or $z''=1+y'$ and $x'+y'=0$.
\item \label{i:xyz3}
If $z=x$, $x+y = 1$, then 
\[
\begin{cases}  
x'=y'' & \text{if } z < \frac{\alpha}{\alpha+1}, \\  
z'=x',\, (x'+1)(y'+1)=1 & \text{if } \frac{\alpha}{\alpha+1} \le z \le \frac{1}{2}, \\
z''=y',\, (x'+1)(y'+1)=1,& \text{if } \frac{1}{2} < z \le \frac{1}{\alpha+1}, \\
x''=y' & \text{if } z > \frac{1}{\alpha+1}.
\end{cases}
\] 
\end{enumerate}
\end{lemma}

\begin{figure}[ht]
\centering
\includegraphics[scale=.20]{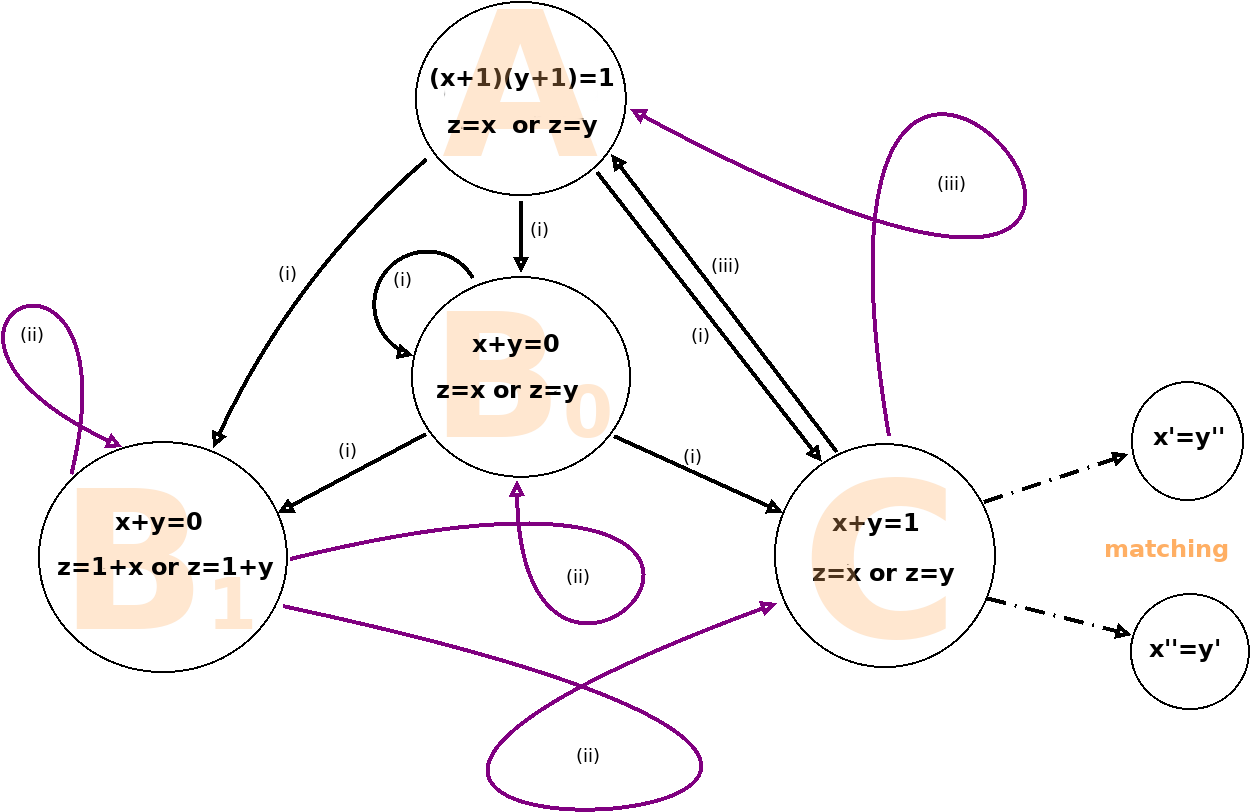}
\caption{A diagram summing up the content of Lemma~\ref{l:xyz} (cf.\ Figure~\ref{fig:diagsimple}). 
Each state is characterized by a condition which is  symmetric in $x$ and $y$, together with a binary option for $z$. Passing through a transition represented by a twisted arrow the coordinate $z$ is  iterated over $T_1^2$ instead of $T_1$ a single time and  $x$ and $y$ switch role instead of being preserved.} \label{fig:diagsimple4}
\end{figure}

\begin{proof}
This is an almost immediate consequence of Lemma~\ref{l:alphabeta}.
Note that $z-x = 1$ implies $z \ge \alpha$, and $y+z = 0$ implies $z \le 1-\alpha$. 
Therefore, we cannot have $z'-x' = 1$ and $y'+z' = 0$ in case~\ref{i:xyz1}.
Similarly, we cannot have $x'+z'' = 0$ and $z''-y' = 1$ in case~\ref{i:xyz2}.
In case~\ref{i:xyz3}, we have $1-\alpha < z < \alpha$; if $\frac{1}{2} < z \le \frac{1}{\alpha+1}$, then $z' \ge \alpha$ and thus $z'' \le \frac{1}{\alpha}-1$, hence we cannot have $z''-y'=1$. 
\end{proof}
Of course, one can duplicate each claim of Lemma~\ref{l:xyz} by switching the roles of $x$ and $y$ both in the hypotheses and in the thesis, see Figure~\ref{fig:diagsimple4}. We use now the notation
\begin{equation} \label{e:xyzn}
\tilde{x}_n = T_\alpha^n\big(\alpha^{(-1)^{r_n+1}}-1\big),\ \tilde{y}_n =
T_\alpha^n\big(\alpha^{(-1)^{r_n}}-1\big),\ \tilde{z}_n = T_1^{n+r_n+1}(\alpha),
\end{equation}
with 
\[
r_n = \#\bigg\{0\le k < n:\, \begin{array}{ll}\tilde{z}_k=1+\tilde{x}_k, \\ \tilde{x}_k+\tilde{y}_k=0,\end{array} \ \mbox{ or}\ \begin{array}{ll}\tilde{z}_k=\tilde{x}_k \in (\tfrac{1}{2}, \tfrac{1}{\alpha+1}], \\ \tilde{x}_k+\tilde{y}_k = 1\end{array}\bigg\}.
\]
Then we have $\tilde{x}_0 = \tilde{z}_0 = \frac{1}{\alpha}-1$, $\tilde{y}_0 = \alpha-1$.
We show that each pre-matching triple $(\tilde{x}_n,\tilde{y}_n,\tilde{z}_n)$ satisfies one of the relations
\[
\begin{array}{lll}
\mathrm{(\tilde{A})} & \tilde{z}_n=\tilde{x}_n, & (\tilde{x}_n +1)(\tilde{y}_n +1) = 1,\\
\mathrm{(\tilde{B}_0)} & \tilde{z}_n=\tilde{x}_n,& \tilde{x}_n+\tilde{y}_n = 0,\\
\mathrm{(\tilde{B}_1)} & \tilde{z}_n=1+\tilde{x}_n,& \tilde{x}_n+\tilde{y}_n=0,\\
\mathrm{(\tilde{C})} & \tilde{z}_n=\tilde{x}_n,& \tilde{x}_n+\tilde{y}_n = 1.
\end{array}
\]

\begin{lemma} \label{l:xyzn} 
Let $\alpha \in (g,1]$, $n \ge 0$.
If the triple $(\tilde{x}_n,\tilde{y}_n,\tilde{z}_n)$ is in state $\mathrm{(\tilde{A})}$, $\mathrm{(\tilde{B}_0)}$ or $\mathrm{(\tilde{B}_1)}$, then  $(\tilde{x}_{n+1},\tilde{y}_{n+1},\tilde{z}_{n+1})$ is in state $\mathrm{(\tilde{B}_0)}$, $\mathrm{(\tilde{B}_1)}$ or $\mathrm{(\tilde{C})}$. 
If $(\tilde{x}_n,\tilde{y}_n,\tilde{z}_n)$ is in state $\mathrm{(\tilde{C})}$, then we have $T_\alpha(\tilde{x}_n)=T_\alpha^2(\tilde{y}_n)$ when $\tilde{z}_n < \frac{\alpha}{\alpha+1}$, $T_\alpha^2(\tilde{x}_n)=T_\alpha(\tilde{y}_n)$ when $z > \frac{1}{\alpha+1}$, and $(\tilde{x}_{n+1},\tilde{y}_{n+1},\tilde{z}_{n+1})$ is in state~$\mathrm{(\tilde{A})}$ when $\frac{\alpha}{\alpha+1} \le \tilde{z}_n \le \frac{1}{\alpha+1}$.

If the triple $(\tilde{x}_n,\tilde{y}_n,\tilde{z}_n)$ is in state $\mathrm{(\tilde{A})}$, $\mathrm{(\tilde{B}_0)}$, $\mathrm{(\tilde{B}_1)}$ or $\mathrm{(\tilde{C})}$, then we have
\begin{align*}
r_{n+1} - r_n & = \begin{cases}
0 & \mbox{if $\mathrm{(\tilde{A})}$, $\mathrm{(\tilde{B}_0)}$, or $\mathrm{(\tilde{C})}$  and $\tilde{z}_n \notin (\tfrac{1}{2}, \tfrac{1}{\alpha+1}]$}, \\
1 & \mbox{if $\mathrm{(\tilde{B}_1)}$, or $\mathrm{(\tilde{C})}$ and $\tilde{z}_n \in (\tfrac{1}{2}, \tfrac{1}{\alpha+1}]$},\end{cases} \\
& = \#\{n+r_n+1 \le k \le n+r_{n+1}+1:\, T_1^k(\alpha) \ge \alpha\}.
\end{align*}
\end{lemma}

\begin{proof}
The first equation for $r_{n+1} - r_n$ follows from the definition.
The second equation means that $\tilde{z}_n < \alpha$ when $r_{n+1} = r_n$, and either $\tilde{z}_n < \alpha$ or $T_1(\tilde{z}_n) < \alpha$ when $r_{n+1} = r_n+1$.
Indeed, $\mathrm{(\tilde{A})}$, $\mathrm{(\tilde{B}_0)}$ and $\mathrm{(\tilde{C})}$ imply that $\tilde{z}_n < \alpha$, $\mathrm{(\tilde{B}_1)}$ implies that $\tilde{z}_n \ge \alpha$ and thus $T_1(\tilde{z}_n) <  \alpha$, $\tilde{z}_n \in (\tfrac{1}{2}, \tfrac{1}{\alpha+1}]$ implies that $T_1(\tilde{z}_n) \ge \alpha$. 

If $r_{n+1}=r_n$, then $T_\alpha(\tilde{x}_n)=\tilde{x}_{n+1}$, $T_\alpha(\tilde{y}_n)=\tilde{y}_{n+1}$ and $T_1(\tilde{z}_n) = \tilde{z}_{n+1}$.
If $r_{n+1}=r_n+1$, then $T_\alpha(\tilde{x}_n)=\tilde{y}_{n+1}$, $T_\alpha(\tilde{y}_n)=\tilde{x}_{n+1}$ and $T_1^2(\tilde{z}_n) = \tilde{z}_{n+1}$. 
Therefore, the relations for $(\tilde{x}_n,\tilde{y}_n,\tilde{z}_n)$ follow from Lemma~\ref{l:xyz}. 
\end{proof}

Recall that 
\[
P_n = P_n(\alpha) = \min\{k \ge 0:\, T_1^{n-k}(\alpha) \le 1-\alpha \ \mbox{or}\ T_1^{n-k-1}(\alpha) \ge \alpha\}.
\]

\begin{lemma} \label{l:odd}
Let $\alpha \in (g,1]$. 
If $\mathrm{(\tilde{C})}$ holds for some $m \ge 1$ and there is no $n < m$ with $\mathrm{(\tilde{C})}$ and $\tilde{z}_n \in
(1-\alpha, \frac{\alpha}{\alpha+1}) \cup (\frac{1}{\alpha+1}, \alpha)$, then $P_{m+r_m+1}$ is odd.
\end{lemma}

\begin{proof}
Let $z_n = T_1^n(\alpha)$ for $n \ge 0$. 
Since $(\tilde{\mathrm{C}})$ holds for $n = m$, we have $\tilde{z}_m=z_{m+r_m+1} > 1-\alpha$ and $(\tilde{\mathrm{A}})$, $(\tilde{\mathrm{B}}_0)$ or $(\tilde{\mathrm{B}}_1)$ for $n = m-1$. 
In case $(\tilde{\mathrm{B}}_0)$, we have $\tilde{z}_{m-1}=z_{m+r_{m-1}+1}=z_{m+r_m} \le 1-\alpha$, thus $P_{m+r_m+1} =1$. 
In case $(\tilde{\mathrm{B}}_1)$, we have $\tilde{z}_{m-1}=z_{m+r_{m-1}+1}=z_{m+r_m-1} \ge \alpha$ and so $z_{m+r_m} < \alpha$, thus $P_{m+r_m+1} =1$. 
Assume now that $(\tilde{\mathrm{A}})$ holds for $n=m-1$.
Then we have $\tilde{z}_{m-1}=z_{m+r_{m-1}+1}=z_{m+r_m} < \alpha$, thus $P_{m+r_m+1} \ge 1$, and $(\tilde{\mathrm{C}})$ holds for $n=m-2$.
If $\frac{1}{2} < \tilde{z}_{m-2} \le \frac{1}{\alpha+1}$, then $z_{m+r_m-1} \ge \alpha$, thus $P_{m+r_m+1} =1$. 
If $\frac{\alpha}{\alpha+1} \le \tilde{z}_n \le \frac{1}{2}$, then $P_{m+r_m+1} = 1$ or $P_{m+r_m+1} = P_{m+r_m-1}+2 = P_{m-2+r_{m-2}+1}+2$. 
As $(\tilde{\mathrm{C}})$ holds for $n=m-2$, we obtain inductively that $P_{m+r_m+1}$ is odd.
\end{proof}

We can now prove that the sets in~\eqref{e:bifurcalpha} and~\eqref{e:bifurc1} are equal; we also obtain an alternative statement of~\eqref{e:bifurc1}. 

\begin{lemma} \label{l:bifurcalpha1}
We have
\begin{align*}
\TEE & = \big\{\alpha \in [g,1]:\, T_1^n(\alpha) \notin (\tfrac{1}{\alpha+1}, \alpha) \ \mbox{for all}\ n \ge 2\ \mbox{and} \\
& \hspace{6.7em} T_1^n(\alpha) \notin (1-\alpha, \tfrac{\alpha}{\alpha+1})\ \mbox{for all}\ n \ge 2\ \mbox{such that $P_n$ is odd}\big\} \\
& = \big\{\alpha \in [g,1]:\, T_1^n(\alpha) \notin (1-\alpha, \tfrac{\alpha}{\alpha+1}) \cup (\tfrac{1}{\alpha+1}, \alpha) \\ 
& \hspace{6.7em} \mbox{for all}\ n \ge 2\ \mbox{such that $P_n$ is odd}\big\}.
\end{align*}
\end{lemma}

\begin{proof}
Let $x_n, y_n$ be as in \eqref{eq:xix}, $z_n = T_1^n(\alpha)$, and $\tilde{x}_n,\tilde{y}_n,\tilde{z}_n$ as in~\eqref{e:xyzn}. 
Then the unordered pair $\{x_n,y_n\}$ is equal to $\{\tilde{x}_n,\tilde{y}_n\}$ for all $n \ge 0$.

Suppose first that $\max(x_m,y_m) > \frac{1}{\alpha+1}$ for some $m \ge 0$, and let $m$ be minimal with this property. 
Then $\tilde{x}_m + \tilde{y}_m = x_m + y_m = 1$ by Lemma~\ref{lem:beforematching}, thus $\mathrm{(\tilde{C})}$ holds for $n=m$ and $\tilde{z}_m \in (1-\alpha, \frac{\alpha}{\alpha+1}) \cup (\frac{1}{\alpha+1}, \alpha)$.
By Lemma~\ref{l:odd}, $P_{m+r_m+1}$ is odd.
Therefore, the sets on the right hand side are contained subsets of~$\TEE$.

For the opposite inclusions, assume that, for some $n\ge 2$, $\frac{1}{\alpha+1} < z_n < \alpha$, or $1-\alpha < z_n < \frac{\alpha}{\alpha+1}$ and $P_n$ is odd, and let $n$ be minimal with this property.
Then there is some $m \ge 1$ such that $n=m+r_m+1$, i.e., $z_n = \tilde{z}_m$. Indeed, if there was no such $m$, then we had $n = m+r_m+2$ and $r_{m+1}-r_m=1$ for some~$m$.
Then $\mathrm{(\tilde{B}_1)}$ would hold for~$m$, or $\mathrm{(\tilde{C})}$ would hold for~$m$ and $\frac{1}{2} < \tilde{z}_m \le \frac{1}{\alpha+1}$.
We cannot have $\frac{1}{2} < \tilde{z}_m \le \frac{1}{\alpha+1}$ since this would imply $T_1(\tilde{z}_m) = z_n \ge \alpha$.
In case $\mathrm{(\tilde{B}_1)}$ for~$m$, we have $\tilde{z}_m \ge \alpha$ and $T_1(\tilde{z}_m) = z_n \le \frac{1}{\alpha}-1 < \frac{1}{\alpha+1}$, thus $P_n = 0$, also contradicting the assumptions on~$z_n$. 
Therefore, we have $z_n = \tilde{z}_m$.

If $\frac{1}{\alpha+1} < \tilde{z}_m < \alpha$, then we have $(\tilde{\mathrm{C}})$ for~$m$, thus $\tilde{x}_m = \tilde{z}_m$, hence $\alpha \not\in \TEE$.
If $1-\alpha < \tilde{z}_m < \frac{\alpha}{\alpha+1}$, then we have $\mathrm{(\tilde{A})}$ or $\mathrm{(\tilde{C})}$ for~$m$. 
In case~$\mathrm{(\tilde{C})}$, we have $\tilde{y}_m = 1-\tilde{z}_m > \frac{1}{\alpha+1}$, hence $\alpha \not\in \TEE$.
Suppose finally that $\mathrm{(\tilde{A})}$ holds for~$m$ and $P_n$ is odd.
Then we have $\mathrm{(\tilde{C})}$ for $m-1$, with $\frac{\alpha}{\alpha+1} \le \tilde{z}_{m-1} \le \frac{1}{2}$ because $P_n \ne 0$.
Then $P_{m-1+r_{m-1}+1} = P_{m+r_m} = P_{n-1}$ is odd by Lemma~\ref{l:odd}; since $P_n = P_{n-1}+1$, this contradicts that $P_n$ is odd.
Therefore, $\TEE$~is contained in the sets on the right hand side.
\end{proof}

\begin{proof}[Proof of Theorem \ref{thm:charactEIT}]
In the proof of Theorem \ref{thm:matchingae} we have already shown that $\EIT \cap [g,1] = \TEE$.
By Lemmas~\ref{lem:X} and~\ref{l:bifurcalpha1}, we obtain the other equations.
\end{proof}

We can also describe the matching exponents and matching intervals in terms of regular continued fractions. 
Here, the \emph{pseudocenter} of an interval denotes the rational number with smallest denominator contained in the interval. 
We write $\overline{a_1,\dots,a_n}$ for the periodic sequence $a_1,a_2,\dots$ with $a_{i+n}=a_i$ for all $i \ge 1$. 

\begin{proposition} \label{p:matching1}
Let $\alpha = [0;1,a_2,a_3,\ldots\,] \in (g,1] \setminus \EIT$.
Let $n \ge 2$ be minimal such that $1-\alpha < T_1^n(\alpha) < \frac{\alpha}{\alpha+1}$ and $P_n(\alpha)$ is odd, or $\frac{1}{\alpha+1} < T_1^n(\alpha) < \alpha$.
Then the matching interval containg $\alpha$ has the endpoints 
\[
\begin{cases}[0;1,\overline{a_2,\ldots,a_n,2}] \ \mbox{and}\  [0;1,a_2,\overline{a_3,\ldots,a_n,a_2{+}1}] & \mbox{if}\ 1-\alpha < T_1^n(\alpha) < \frac{\alpha}{\alpha+1}, \\ [0;\overline{1,a_2,\ldots,a_n}]\ \mbox{and} \ [0;\overline{1,a_2,\ldots,a_n,1}] & \mbox{if}\ \frac{1}{\alpha+1} < T_1^n(\alpha) < \alpha,\end{cases}
\]
the matching exponents
\[
\begin{cases}M = N-2 = n-R_n & \mbox{if}\ (-1)^{R_n}\,(T_1^n(\alpha)-\frac{1}{2}) > 0, \\  M = N = n-R_n+1 & \mbox{if}\ (-1)^{R_n}\,(T_1^n(\alpha)-\frac{1}{2}) < 0,\end{cases} 
\]
with $R_n = \#\{1 \le k < n:\, T_1^k(\alpha) \ge \alpha\}$, and the pseudocenter
\[
\begin{cases} [0;1,a_2,\ldots,a_n,2,a_2,\dots,a_L,1] & \mbox{if}\ 1-\alpha < T_1^n(\alpha) < \frac{\alpha}{\alpha+1}, \\ [0;1,a_2,\ldots,a_n,1,a_2,\dots,a_L,2] & \mbox{if}\ \frac{1}{\alpha+1} < T_1^n(\alpha) < \alpha,\end{cases}
\]
with $L = \min\{k \ge 1:\, a_{k+1} \ne 1\}$.
\end{proposition}

\begin{proof}
Let us first prove the formulae for the matching exponents. From the proof of Lemma~\ref{l:bifurcalpha1}, we know that $n = m+r_m+1$ for some $m \ge 0$, with $\tilde{x}_m=\tilde{z}_m$ and $\tilde{y}_m+\tilde{z}_m = 1$, thus by Lemma~\ref{l:xyzn} we find $T_\alpha(\tilde{x}_m) = T_\alpha^2(\tilde{y}_m)$ when $1-\alpha < z_n < \frac{\alpha}{\alpha+1}$, $T_\alpha^2(\tilde{x}_m) = T_\alpha(\tilde{y}_m)$ when $\frac{1}{\alpha+1} < z_n < \alpha$. 
We have $\tilde{x}_m = x_m$, $\tilde{y}_m = y_m$ if $r_m$ is even, $\tilde{x}_m = y_m$, $\tilde{y}_m = x_m$ if $r_m$ is odd.
Therefore the matching exponents are $M=N=m+2$ if $r_m$ is even and $z_n < \frac{\alpha}{\alpha+1} \le 1/2$, or if $r_m$ is odd and $z_n > \frac{1}{\alpha+1} \ge 1/2$, i.e., when $(-1)^{r_m} (z_n-\frac{1}{2}) < 0$.
Similarly, the matching exponents are $M=m+1$, $N=m+3$, when $(-1)^{r_m} (z_n-\frac{1}{2}) > 0$.
By Lemma~\ref{l:xyzn}, we have 
\[
r_m = r_m - r_0 = \#\{1 \le k \le m+r_m:\, T_1^k(\alpha) \ge \alpha\} = R_n,
\]
thus $n=m+R_n+1$. This gives the formulae for the matching exponents.

For the endpoints of the intervals, Proposition~\ref{p:atmatching} implies that we have to determine $\tilde{\alpha}$ with the same first $n$ partial quotients as~$\alpha$ and $T_1^n(\tilde{\alpha}) = \tilde{\alpha}$, $T_1^n(\tilde{\alpha}) = \frac{1}{\tilde{\alpha}+1}$, $T_1^n(\tilde{\alpha}) = \frac{\tilde{\alpha}}{\tilde{\alpha}+1}$, $T_1^n(\tilde{\alpha}) = 1-\tilde{\alpha}$ respectively. 
Since 
\begin{equation} \label{e:conv}
\begin{split}
\frac{1}{\alpha+1} & = [0;1,1,a_2,a_3,\ldots], \quad \frac{\alpha}{\alpha+1} = \frac{1}{1+\frac{1}{\alpha}} =  [0;2,a_2,a_3,\ldots], \\ 
1-\alpha & = \frac{1}{1+\frac{1}{\frac{1}{\alpha}-1}} = [0;a_2{+}1,a_3,a_4,\ldots],
\end{split}
\end{equation}
we have
\begin{align*}
T_1^n([0;\overline{1,a_2,\ldots,a_n}]) & = [0;\overline{1,a_2,\ldots,a_n}], \\
T_1^n([0;\overline{1,a_2,\ldots,a_n,1}]) & = [0;\overline{1,1,a_2,\ldots,a_n}] = \frac{1}{[0;\overline{1,a_2,\ldots,a_n,1}]+1}, \\
T_1^n([0;1,\overline{a_2,\ldots,a_n,2}]) & = [0;\overline{2,a_2,\ldots,a_n}] = \frac{[0;1,\overline{a_2,\ldots,a_n,2}]}{[0;1,\overline{a_2,\ldots,a_n,2}]+1}, \\
T_1^n([0;1,a_2,\overline{a_3,\ldots,a_n,a_2{+}1}]) & = [0;\overline{a_2{+}1,a_3,\ldots,a_n}] \\ 
& = 1 - [0;1,a_2,\overline{a_3,\ldots,a_n,a_2{+}1}].
\end{align*}

If $\frac{1}{\alpha+1} < T_1^n(\alpha) < \alpha$, then the regular continued fraction expansion of any $\tilde{\alpha}$ in the matching interval starts with $1,a_2,\dots,a_n,1,a_2,\dots,a_L$.
Note that $L < n$ and $L$ is odd since $\alpha > g$.
Therefore, we have 
\begin{gather*}
[0;1,a_2,\ldots,a_L,1] < [0;\overline{1,1,a_2,\ldots,a_n}] < [0;1,a_2,\ldots,a_L,2] \\
< [0;\overline{1,a_2,\ldots,a_n}] < [0;1,a_2,\ldots,a_L],
\end{gather*}
which implies that $[0;1,a_2,\ldots,a_n,1,a_2,\dots,a_L,2]$ is the rational number with smallest denominator in the matching interval of~$\alpha$. 

Let now $1-\alpha < T_1^n(\alpha) < \frac{\alpha}{\alpha+1}$.
If $a_2=1$, then the regular continued fraction expansion in the matching interval starts with $1,a_2,\dots,a_n,2,a_3,\dots,a_L$. 
Since
\begin{gather*}
[0;2,a_3,\dots,a_L] < [0;\overline{2,a_3,\ldots,a_n}] < [0;2,a_3,\dots,a_L,2] \\
= [0;2,a_2,\dots,a_L,1] < [0;\overline{2,a_2,\ldots,a_n}] < [0;2,a_2,\dots,a_L],
\end{gather*}
the pseudocenter is $[0;1,a_2,\ldots,a_n,2,a_2,\dots,a_L,1]$.
If $a_2 \ge 2$, then $L=1$ and 
\[
0 < [0;\overline{a_2{+}1,a_3,\ldots,a_n}] < [0;3] = [0;2,1] < [0;\overline{2,a_2,\ldots,a_n}] < [0;2],
\]
thus the pseudocenter is $[0;1,a_2,\ldots,a_n,2,1]$.
\end{proof}

\section{Dimensional results for $\EE$ }\label{sec:dim}

Now that we established several characterisations of $\EIT$ we will focus on dimensional results  of $\EIT$ in this section. We make use of two sets and the following proposition.

\begin{proposition}\label{p:FnCn} 
Let us consider the sets\footnote{The sets $H_n$ are sometimes referred to as \emph{high type} numbers.}
\begin{align*}
H_n & =\{x\in[0,1]\,:\,  x=[0;a_1,a_2,\ldots] \text{ such that } a_k \geq n \text{ for all } k\ge 1\}, \\
C_n & = \{x\in[0,1]\,:\,  x=[0;a_1,a_2,\ldots] \text{ and } a_k a_{k+1} \cdots a_{k+n-1} \neq 1 \text{ for all } k\ge 1 \},
\end{align*}
where $[0;a_1,a_2,\ldots]$ denotes the regular continued fraction.
For these sets we have $\dim_H(H_n)>\frac{1}{2}$, $\dim_H(C_n)<1$, and $\lim_{n \to \infty} \dim_H(C_n)= 1$.
\end{proposition}

\begin{proof}
In~\cite{G} it is shown that $\dim_H(H_n)>\frac{1}{2}+\frac{1}{2\log(n+2)}$ for $n>20$. Since $H_{n+1}\subset H_n$ we find that $\dim_H(H_n)>\frac{1}{2}$ for all $n\in\mathbb{N}$.

Let $BAD(g)= \{x\in [0,1]: g\not\in \overline{\{T_1^k(x) :k\in\mathbb{N}\}}\}$. 
Then~\cite{HY} gives us that $BAD(g)$ is $\alpha$-winning and therefore it has Hausdorff dimension~1. 
On the other hand, it is not difficult to check that  $BAD(g)=\bigcup_n  C_n$ and since $C_n$ is an increasing sequence of sets we get
\[
1=\sup_n \dim_H(C_n)= \lim_{n\to \infty}  \dim_H(C_n). 
\]

To prove that $\dim_H(C_n)<1$, let us point out that $C_n\subset K_n$ where 
\[
K_n: = \{x\in[0,1]\,:\,  x=[0;a_1,a_2,\ldots] \text{ and } a_{kn+1} \cdots a_{kn+n} \neq 1 \text{ for all } k\ge 0 \}.
\]
It is then sufficient to show that $\dim_H(K_n)<1$ for all $n \ge 1$.

Note that the set $K_n$ can be easily described in terms of the map $T:=T_1^n$, which is a map with countable many full branches; indeed, $K_n$ is the set of points whose iterates under $T$ never enter the domain of the branch of $T$ containing the fixed point~$g$, i.e., the interval between $F_{n-1}/F_n$ and $F_n/F_{n+1}$, where $F_n$ are the Fibonacci numbers $F_0=F_1=1$, $F_{n+1}=F_n+F_{n-1}$.

Equivalently the set $K_n$ can also be described as the limit set of an Iterated Function System induced by the family~$\mathcal{S}'$ of all  but one inverses branches of~$T$; this setting allows us to use Theorem 4.7 of \cite{MU96} to conclude that the Hausdorff dimension of the limit set $K_n$ induced by the family $\mathcal{S}'$ is strictly smaller than the dimension of the IFS generated by the family $\mathcal{S}$ of {\bf all} inverse branches of~$T$, yielding that $\dim_H(K_n)<1$; for a more detailed discussion of the general results of \cite{MU96} in the setting of continued fractions we also suggest  the paper \cite{MU99}.
\end{proof}

We can now prove Theorem~\ref{th:aroundg}.

\begin{proof}[Proof of Theorem~\ref{th:aroundg}]
Define $f_a:[0,1]\rightarrow[0,1]$ by $f_a(x)=\frac{1}{a+x}$ with $a\in\mathbb{N}$, and let $\alpha \in \hat{C_n}:=f_1^{2n+3}\circ f_2(C_{2n+1})$.
Then the RCF expansion of $\alpha$ starts with exactly $2n+3$ ones and has no other occurrence of $2n+1$ consecutive ones.
Therefore, we have $\alpha \in (g,1)$.
Suppose that $\alpha \notin \EIT$. 
Then we have by Theorem~\ref{thm:charactEIT} some $k \ge 2$ such that $T_1^k(\alpha) \in (\frac{1}{1+\alpha},\alpha) \cup (1-\alpha,\frac{\alpha}{1+\alpha})$.
This implies that $T_1^{k+1}(\alpha) \in (T_1(\alpha),\alpha) \cup (T_1(\alpha), T_1^2(\alpha))$, hence the RCF expansion of $T_1^{k+1}(\alpha)$ starts with $2n+1$ ones, contradicting that $\alpha\in \hat{C}_n$.
Hence we have 
\[
\hat{C}_n \subset \EIT \cap [g,1].
\] 

Since $f_a$ is bi-Lipschitz for all $a\in\mathbb{N}$, the same is true for any finite composition of these maps. Since bi-Lipschitz maps preserve the Hausdorff dimension, we have $\dim_H(\hat{C_n})=\dim_H(C_{2n+1})$.
Then it follows from Proposition~\ref{p:FnCn} that
\[
\dim_H(\EIT) \ge \dim_H\bigg(\bigcup_{n\ge1} \hat{C}_n\bigg)=\sup_{n\ge1} \dim_H(\hat{C}_n)=\sup_{n\ge1} \dim_H(C_{2n+1})=1,
\]
thus $\dim_H(\EIT)=1$. 
For any $\delta>0$, we have that $\hat{C}_n \subset (g, g+\delta)$ for all sufficiently large~$n$, and so $\dim_H(\EIT \cap (g,g+\delta))=1$.

Let now $n$ be such that $\frac{F_{2n}}{F_{2n+1}} \le g+\delta$. 
For $\alpha \in\EIT \cap (g+\delta,1)$, we have $T_1^k(\alpha) \notin (\frac{1}{1+\alpha}, \alpha)$ for all $k \ge 0$, hence the RCF expansion of $\alpha$ contains no $2n+1$ consecutive ones, thus $\EIT \cap (g+\delta,1) \subset C_{2n+1}$. 
By Proposition~\ref{p:FnCn}, this implies that $\dim_H(\EIT \cap (g+\delta,1)) < 1$. 
\end{proof}

To get results on the Hausdorff dimension around a point $b\in \EIT\cap \mathbb{Q}$ we need more insight in the behavior around such a point. We establish this with the following lemma.

\begin{lemma}\label{l:aroundbad}
If $\alpha_0\in \EIT\cap \QQ\cap (g,1)$ has RCF expansion $\alpha_0=[0;1,a_2,...,a_k]$, then there is a $c \in\mathbb{N}$ such that $[0;1,a_2,...,a_k,c, c_1,c_2, \dots] \in \EIT$ for all finite or infinite sequences $c_1,c_2,\ldots$ with $c_j \ge a_2+1$ for all $j \ge 1$.
\end{lemma}

\begin{proof}
Let $\alpha_0\in \EIT\cap \QQ\cap (g,1]$ and $\alpha = [0;1,a_2,...,a_k,c, c_1,c_2, \dots]$ with $c = \max(a_2,\dots,a_k)+2$, $c_j \ge a_2+2$ for all $j \ge 1$. 
By Theorem~\ref{thm:charactEIT}, we have to show that $T_1^j(\alpha) \notin (\frac{1}{1+\alpha},\alpha)$ for all $j \ge 2$ and $T_1^j(\alpha) \notin (1-\alpha,\frac{\alpha}{1+\alpha})$ for all $j \ge 2$ such that $P_j(\alpha)$ is odd. 
By \eqref{e:conv}, we have $T_1^j(\alpha) \le \frac{1}{a_2+2} \le 1-\alpha$ for all $j \ge k$. 
For $j<k$, we get from $\max(a_2,\dots,a_k) < c < \infty$ that $T_1^j(\alpha) \in (\frac{1}{1+\alpha},\alpha)$ if and only if $T_1^j(\alpha_0) \in (\frac{1}{1+\alpha_0},\alpha_0)$, $T_1^j(\alpha) \notin (1-\alpha,\frac{\alpha}{1+\alpha})$ if and only if $T_1^j(\alpha_0) \notin (1-\alpha_0,\frac{\alpha}{1+\alpha_0})$, and $P_j(\alpha) = P_j(\alpha_0)$. 
Therefore, $\alpha_0\in \EIT$ implies that $\alpha \in \EIT$. 
\end{proof}

Now Theorem~\ref{th:bifrat} follows almost directly.

\begin{proof}[Proof of Theorem~\ref{th:bifrat}]
The fact that there are infinitely many rationals in $\EIT$ is given by the fact that $\frac{n-1}{n}\in \EIT$ for all $n \geq 3$. 
Furthermore, $\EIT\cap \QQ\cap (g,1]$ has no isolated points since in the proof Lemma~\ref{l:aroundbad} one can take  $c$ arbitrarily large. 
For the dimensional result, we reason as follows. 
The composition  $f_{a_1}\circ\ldots\circ f_{a_k}$ is bi-Lipschitz. 
Furthermore, from Lemma~\ref{l:aroundbad} it follows that $f_{a_1}\circ\ldots\circ f_{a_k}(H_n) \subset  \EIT$ for all sufficiently large~$n$. 
Using Proposition~\ref{p:FnCn}, the theorem now follows.
\end{proof}

\section{Remarks and open questions}

\begin{enumerate}
\item Using the techniques of \cite{T} one can prove that  the entropy is H\"older continuous also for the family (TI); it is natural to ask whether it is Lipschitz continuous. Let us mention that the answer to this  question is negative for the cases (N) and (KU), but this is due to the failure of  the Lipschitz property  at points accumulated by matching intervals with arbitrarily high matching index.
\item  Is the entropy weakly decreasing on $[g,1]$? We believe the answer is affirmative, but we cannot rule out some devil staircase pathology (unless we prove that the entropy is Lipschitz, or at least absolutely continuous).
Recently Nakada~\cite{N19} proved that the entropy in case (N) attains its maximum on the central plateau,  the same methods might be useful to deal also with this question.
\item  Can one characterize the isolated points of $\EE$?
The set $\EE$ certainly contains countable many isolated points, for instance  for each $a\geq 2$  the value $\gamma_a:=[0;\overline{1,a,1}]$   is the separating element between the two adjacent matching intervals $(\beta_a,\gamma_a)$ and $(\gamma_a,\eta_a)$ with $\beta_a:=[0;\overline{1,a,1,1,a}]$ and
 $\eta_a:=[0;\overline{1,a}]$. But is  there some countable chain of adjacent intervals as it was observed for the family (N) (see \textsection 3.3 in \cite{CT1})? 
\item Are there non isolated points of $\EIT$ at which the local Hausdorff dimension of $\EIT$ falls in the open interval $(0,1/2)$?
\item Let  $\EIT_0$ denote the nonisolated points of $\EIT$; by Theorem~\ref{th:bifrat} we know that $\EIT\cap \QQ\subset \EIT_0$; is it true that $\overline{\EIT\cap \QQ}=\EIT_0$?
We think that the answer to this last question is affirmative (which, by Theorem \ref{th:bifrat}, would imply that the answer to the previous question is negative). 
\end{enumerate}

\section*{Acknowledgements}
The first author acknowledges the support of MIUR PRIN Project Regular and stochastic behavior in dynamical systems nr.\ 2017S35EHN and of the GNAMPA group of the “Istituto Nazionale di Alta Matematica” (INdAM).
The third author was supported by the Agence Nationale de la Recherche through the project Codys (ANR--18--CE40--0007).

\bibliographystyle{alpha}
\bibliography{articleito}
\end{document}